\documentclass[psamsfonts]{amsart}
 
\usepackage{amssymb,amsfonts}
\usepackage[all,arc]{xy}
\usepackage{enumerate}
\usepackage{mathrsfs}
\usepackage{lipsum}
\usepackage{hyperref}

\usepackage[left=1.3 in,top=1.3 in,right=1.3 in,bottom=1.3 in]{geometry}

\usepackage{fancyhdr} 
\newtheorem{thm}{Theorem}[section]

\newtheorem{lem}[thm]{Lemma}

\theoremstyle{definition}

\theoremstyle{remark}

\newcommand{\dif}{ \,\mathrm{d} }

\makeatletter
\let\c@equation\c@thm
\makeatother

\numberwithin{equation}{section}

\title{Limiting value of higher Mahler measure}

\author[A. Biswas]{Arunabha Biswas}
\address{DEPARTMENT OF MATHEMATICS AND STATISTICS, TEXAS TECH UNIVERSITY, LUBBOCK, TX 79409, USA}
\email{arunabha.biswas@ttu.edu}
\author[C. Monico]{Chris Monico}
\address{DEPARTMENT OF MATHEMATICS AND STATISTICS, TEXAS TECH UNIVERSITY, LUBBOCK, TX 79409, USA}
\email{c.monico@ttu.edu}

\begin{document}

\begin{abstract}

We consider the $k$-higher Mahler measure $m_k(P)$ of a 
Laurent polynomial $P$ as the integral of $\log ^k \left| P \right|$ over 
the complex unit circle. In this paper we derive an explicit formula
for the value of $\left| m_k(P) \right|/k!$ as $k \to \infty.$

\end{abstract}

\subjclass[2010]{11R06; 11M99.}
\keywords{Mahler measure, higher Mahler measure.}


\maketitle




\section{Introduction}

For a non-zero Laurent polynomial $P(z) \in \mathbb{C}[z, z^{-1}],$ 
the $k$-higher Mahler measure of $P$ is defined \cite{R2} as
\[
  m_k(P) = \int _{0}^{1} \log^k \left|P\left(e^{2\pi it}\right)\right|\dif{t}.
\]
For $k=1$ this coincides with the classical (log) Mahler measure defined as
\[
  m(P) = \log|a| + \sum_{j=1}^{n} \log \left(\max \{1,|r_j|\} \right), 
          \,\,\,\mbox{for} \,\,\, P(z)=a \prod_{j=1}^{n} (z-r_j),
\] 
since by Jensen's formula  $m(P)=m_1(P)$ \cite{book}.

Though classical Mahler measure was studied extensively, higher 
Mahler measure was introduced and studied very recently by  
Kurokawa, Lalín and Ochiai \cite{R2}  and Akatsuka \cite{R1}. 
It is very difficult to evaluate $k$-higher Mahler measure for 
polynomials except few specific examples shown in \cite{R1} and 
\cite{R2} , but it is relatively easy to find their limiting values.

In \cite{sinha} Lalin and Sinha answered Lehmer's question 
\cite{book} for higher Mahler measure by finding 
 non-trivial lower bounds for $m_k$ on $\mathbb{Z}[z]$ for $k \geq 2.$

In \cite{B} it has been shown using Akatsuka's zeta function of \cite{R2} that for $|r|=1$,
$|m_k(z+r)|/k! \to 1/\pi$ as 
$k \to \infty.$
In this paper we generalize this result by computing the same limit
for an arbitrary Laurent polynomial 
$P(z) \in \mathbb{C}[z, z^{-1}]$ using a different technique.

\begin{thm} \label{thm:main}
Let $P(z) \in \mathbb{C}\left[z, z^{-1}\right]$ be a Laurent polynomial,
possibly with repeated roots.
Let $z_1,\dots, z_n$ be the distinct roots of $P$. Then 
\[ 
  \lim_{k \to \infty} \frac{\left| m_k(P) \right|}{k!}
         = \frac{1}{\pi} \, \sum_{z_j \in S^1} \frac{1}{\left| P'(z_j) \right|},
\]
where $S^1$ is the complex unit circle $|z|=1$, and the right-hand side is
taken as $\infty$ if $P'(z_j)=0$ for some $z_j\in S^1$, i.e., if $P$ has a 
repeated root on $S^1$.
\end{thm}

\section{Proof of the theorem}
  We first prove several lemmas which essentially show that the integrand
may be linearly approximated near the roots of $P$ on $S^1$.

\begin{lem} \label{lem1}
  Let $P(z) \in \mathbb{C}\left[z, z^{-1}\right] $ be a 
  Laurent polynomial and $A \subseteq [0,1]$ be a closed 
  set such that $P\left(e^{2 \pi i t}\right) \neq 0$ for 
  all $t \in A.$ Then 
  \[
    \lim_{k\to\infty} \frac{1}{k!} \int_A \log^k \left| P\left(e^{2 \pi i t}\right) \right| \d{t}=0
  \]
\end{lem}

\begin{proof}
  Since $A$ is closed, due to the periodicity of $e^{2\pi it}$ 
  and continuity of $P(e^{2\pi it})$ there exist constants $b$ and 
  $B$ such that $0 < b \leq \left| P\left(e^{2\pi it}\right) \right| \leq B$ on $A$. 
  Then for each positive integer $k$, 
  $(\log^k \left| P\left(e^{2\pi it}\right) \right|)/k!$ 
  is bounded between $(\log ^k b)/k!$ and $(\log ^k B)/k!$, and therefore 
  $(1/k!)\int _A\log^k \left| P\left(e^{2\pi it}\right) \right| \d{t}$ 
  is bounded between $(\mu A \, \log ^k b)/k!$ and 
  $(\mu A \, \log ^k B)/k!,$ where $\mu A$ is the Lebesgue measure of 
  $A$. The result follows by letting $k$ tend to infinity.
\end{proof}

\begin{lem} \label{lem:lapprx}
  Let $P(z) \in \mathbb{C}\left[z, z^{-1}\right] $ be a Laurent polynomial
  with a root of order one at $z_0=e^{2\pi it_0},$ and $P'(z)$ be its 
  derivative with respect to $z$. 
  Then for each $\varepsilon\in(0,1)$ there exists $\delta >0$ such that $|t-t_0| < \delta$ implies
  \[
    \left| 2 \pi (1-\varepsilon)(t-t_0)  P' \left( e^{2 \pi i t_0} \right) \right| 
    \leq \left| P\left( e^{2 \pi i t} \right) \right| 
    \leq \left| 2 \pi (1+\varepsilon)(t-t_0)  P' \left( e^{2 \pi i t_0} \right) \right|.
  \]
 \end{lem}

\begin{proof}
  Set $f(t)=P\left(e^{2 \pi i t}\right).$ 
  Then $f'(t_0)=2 \pi i P'\left(e^{2 \pi i t_0}\right)\ne 0$ and 
  \[
    f'(t_0)=\lim_{t \to t_0} \frac{f(t)-f(t_0)}{t-t_0}.
  \]
  Since $f'(t_0)\ne 0$, it follows that for each $\varepsilon\in(0,1)$ 
  there exists $\delta > 0$ such that $0<|t-t_0|<\delta$ implies 
  \[
    1-\varepsilon < \left| \frac{f(t)-f(t_0)}{(t-t_0)} \cdot \frac{1}{f'(t_0)} \right| < 1+\varepsilon,
  \]
  which proves the lemma since $f(t_0)=P(z_0)=0.$
\end{proof}

\begin{lem} \label{lem3}
  Let $c \neq 0,$ and $t_0 \in \mathbb{R}.$ Then for all $\varepsilon > 0,$
  \[
    \lim_{k \to \infty} \frac{1}{k!} \left|\,\,
        \int_{t_0 - \varepsilon}^{t_0 + \varepsilon} \log^k|c(t-t_0)| \d{t} \right|
     = \frac{2}{|c|}.
  \] 
\end{lem}

\begin{proof}
  For $k \geq 1$ and $x>0,$ it follows from integration by parts and induction that
  \[
    \int_{0}^{x} \log^k u \d{u} 
      = x \log ^k x +x \sum_{j=1}^{k} \frac{(-1)^j \, k! \, \log^{k-j} \, x}{(k-j)!}.
  \]
  Using the even symmetry of the integrand and substituting $u=|c(t-t_0)|$, we have 
  \[
    \frac{1}{k!} \left|\,\,\int_{t_0 - \varepsilon}^{t_0 + \varepsilon} \log^k|c(t-t_0)| \d{t} \right| 
     = \frac{2}{|c|\,k!} \left|\,\int_{0}^{|c\varepsilon|} \log^k u \,\d{u}  \right|,
  \] 
  and it follows that
  \begin{eqnarray*}
    \lim_{k \to \infty} \frac{1}{k!} \left|\,\,\int_{t_0 - \varepsilon}^{t_0 + \varepsilon} \log^k|c(t-t_0)| \d{t} \right| 
      &=& \lim_{k \to \infty} \frac{2}{|c|\,k!} \left|\,\int_{0}^{|c\varepsilon|} \log^k u \d{u} \right|\\
      &=& 2 \varepsilon \lim_{k \to \infty} \! \left| \frac{\log^k |c \varepsilon|}{k!} + 
            \sum_{j=1}^{k} \frac{(-1)^j \log^{k-j} |c \varepsilon|}{(k-j)!} \right| \\
      &=& 2 \varepsilon \left| \sum_{n=0}^{\infty} \frac{(-1)^{n} \log^{n} |c \varepsilon|}{n!} \right| \\
      &=& 2 \varepsilon e^{- \log |c \varepsilon|} = 2/|c|.
  \end{eqnarray*}
  \end{proof}

\begin{lem} \label{lem4}
  Let $P(z) \in \mathbb{C}\left[z, z^{-1}\right] $ be a Laurent polynomial 
  with a root of order one at $z_0=e^{2 \pi i t_0}.$ Then for all sufficiently 
  small $\delta > 0,$ 
  \[
    \lim_{k \to \infty} \frac{1}{k!} \left|\,\,
      \int_{t_0-\delta}^{t_0+\delta} \log^k \left| P \left( e^{2 \pi i t} \right) \right| \d{t} \right|
      =\frac{1}{\pi \left| P'\left( e^{2\pi it_0}\right) \right|}.
  \] 
\end{lem}

\begin{proof}
  First notice that since $z_0$ has order one, it cannot be a root of $P'(z).$ 
  Now let $\varepsilon \in (0,1).$ By Lemma \ref{lem:lapprx} there is a $\delta > 0$ 
  such that $|t-t_0| < \delta$ implies
  \[
    \left| 2\pi(1-\varepsilon)(t-t_0) P'\left( e^{2\pi it_0}\right)\right| 
    \leq \left|P\left(e^{2\pi it}\right)\right| 
    \leq \left|2\pi(1+\varepsilon)(t-t_0) P'\left(e^{2\pi it_0}\right)\right| 
    \leq 1.
  \]
    Setting $c=2\pi(1-\varepsilon)P'\left(e^{2\pi it_0}\right)$ 
    and $d=2\pi(1+\varepsilon)P'\left(e^{2\pi it_0}\right)$ it follows that 
    for $0< |t-t_0|<\delta$,
  \[
    \log |c(t-t_0)| \leq \log\left|P\left(e^{2\pi it}\right)\right| 
    \leq \log|d(t-t_0)| \leq 0,
  \]
  and hence 
  \[
    \left|\log^k|c(t-t_0)|\right| \geq \left|\log^k\left|P\left(e^{2\pi it}\right)\right|\right| 
    \geq \left|\log^k|d(t-t_0)| \right| \geq 0,
  \]
 for all $k \in \mathbb{N}.$ Therefore,
  \[
    \int \limits_{t_0-\delta}^{t_0+\delta} \left|\log^k|c(t-t_0)|\right| \d{t} 
    \geq \! \! \! \int \limits_{t_0-\delta}^{t_0+\delta} \left|\log^k\left|P\left(e^{2\pi it}\right)\right|\right|\d{t} 
    \geq \! \! \! \int \limits_{t_0-\delta}^{t_0+\delta} \left|\log^k|d(t-t_0)| \right| \d{t} \geq 0.
  \]
  But on $(t_0-\delta,t_0+\delta),$ for each fixed $k$, 
  either all three functions $\log^k |c(t-t_0)|$, 
  $\log^k \left|P\left(e^{2\pi it}\right)\right|$ and 
  $\log^k |d(t-t_0)|$ are negative (if $k$ is odd), or positive (if $k$ is even). 
  So the integrals of their absolute values are equal to the 
  absolute values of their integrals and therefore we have 
  \[
    \left|\,\,\int \limits_{t_0-\delta}^{t_0+\delta}  \log^k |c(t-t_0)| \d{t} \right|  
    \geq \left|\,\,\int \limits_{t_0-\delta}^{t_0+\delta} \log^k \left|P\left(e^{2\pi it}\right)\right| \d{t}\right|
    \geq \left|\,\,\int \limits_{t_0-\delta}^{t_0+\delta} \log^k |d(t-t_0)| \d{t} \right|.
  \]
  By Lemma \ref{lem3} it follows that
  \[
    \frac{2}{|c|} 
    \geq \lim_{k \to \infty} \frac{1}{k!} \left|\,\,
          \int_{t_0-\delta}^{t_0+\delta} \log^k \left|P\left(e^{2\pi it}\right)\right|\right| 
    \geq \frac{2}{|d|}.
  \]
  Since $c=2 \pi (1-\varepsilon)P' \left( e^{2 \pi i t_0} \right)$ and 
  $d=2 \pi (1+\varepsilon)P' \left( e^{2 \pi i t_0} \right)$ and $\varepsilon > 0$ 
  is arbitrary, we are done.
\end{proof}


  With these lemmas, we now proceed to prove the main theorem.

\begin{proof}[Proof of Theorem \ref{thm:main}]
  First notice that 
  \[
    \frac{m_k(P)}{k!} = \frac{1}{k!} \int_{0}^{1} \log^k \left|P\left(e^{2\pi it}\right)\right|\d{t}.
  \]
  If $P(z)$ does not have any roots on $S^1$ then choosing 
  $A=[0,1]$ and applying Lemma \ref{lem1} we see that 
  $ \left| m_k(P) \right|/k! \to 0$ as $k \to \infty$ 
  and the theorem holds in this case.

  Now let $t_1,\dots, t_m\in[0,1]$ such that 
  $e^{2\pi it_1},\dots, e^{2\pi it_m}$ are the distinct roots of $P$ on $S^1$.
  Let $\delta>0$ be sufficiently small so that $|P(e^{2\pi it_j})|<1$ on
  each interval $(t_j-\delta, t_j+\delta)$, $j=1,\dots, m$, and these intervals are disjoint and define
  \[
    A=[0,1] \smallsetminus \bigcup _{j=1}^{m} (t_j -\delta, t_j + \delta).
  \]
  Using Lemma \ref{lem1}, and the fact that $\log|P(e^{2\pi it})|<0$ on $[0,1]\setminus A$,
  we find that
  \begin{eqnarray}
    \lim_{k\to\infty}\frac{|m_k(P)|}{k!} 
    &=& \lim_{k\to\infty} \frac{1}{k!} \left| \int_A \log^k|P(e^{2\pi it})|\d{t}
                                            + \int_{[0,1]\setminus A} \log^k|P(e^{2\pi it})|\d{t} \right| \nonumber \\
    &=& \lim_{k\to\infty} \sum_{j=1}^m \frac{1}{k!} 
           \left| \int_{t_j-\delta}^{t_j+\delta} \log^k|P(e^{2\pi it})|\d{t} \right| \label{eqn:decomp}
  \end{eqnarray}
  If $P$ has no repeated roots on $S^1$, by Lemma \ref{lem4}, this final sum is equal 
  $\pi^{-1}\sum_{j=1}^m |P'(e^{2\pi it_j})|^{-1}$,
  and so the theorem is proven in this case.

  Finally, if $P$ has a repeated root on $S^1$, we may assume without loss of generality that
  $P(z_1)=P'(z_1)=0$ where $z_1=e^{2 \pi i t_1}$. With $f(t)=P(e^{2\pi it})$, we have that $f(t_1)=f'(t_1)=0$.
  Then for each $\varepsilon\in(0,1)$ there is a $\delta_\varepsilon \in (0,1)$ such that
  \[
    \left| \frac{f(t)}{t-t_1}\right| = \left| \frac{f(t)-f(t_1)}{t-t_1}\right| 
      \le \varepsilon, \hspace{12pt}\mbox{ for all } 0<|t-t_1|<\delta_\varepsilon.
  \]
  It follows that $\log|f(t)|\le \log|\varepsilon(t-t_1)|<0$ for all $0<|t-t_1|<\delta_\varepsilon$,
  and so
  \[
    \left| \log^k|f(t)|\right| \ge \left| \log^k|\varepsilon(t-t_1)|\right|, 
      \hspace{12pt}\mbox{ for all } 0<|t-t_1|<\delta_\varepsilon.
  \]
  We may assume that $\delta_\varepsilon < \delta$, and using \eqref{eqn:decomp} and Lemma \ref{lem3}
  deduce that
  \begin{eqnarray*}
    \lim_{k\to\infty}\frac{|m_k(P)|}{k!}
      &\ge & \lim_{k\to\infty} \left| \int_{t_1-\delta}^{t_1+\delta} \log^k|P(e^{2\pi it})|\d{t} \right| \\
      & =  & \lim_{k\to\infty} \int_{t_1-\delta}^{t_1+\delta} \left| \log^k|P(e^{2\pi it})| \right| \d{t} \\
      &\ge & \lim_{k\to\infty} \int_{t_1-\delta_\varepsilon}^{t_1+\delta_\varepsilon} 
                     \left| \log^k|\varepsilon(t-t_0)| \right| \d{t} \\
      &=& \frac{2}{|\varepsilon|}.
  \end{eqnarray*}
  Since $\varepsilon\in(0,1)$ was arbitrary, the limit in question diverges to $\infty$
  and the theorem is proven.
\end{proof}



\bibliographystyle{elsarticle-num}  

\end{document}